\documentclass[11pt]{amsart}
\usepackage{graphicx}
\usepackage{amsmath}
\usepackage{amsfonts}
\usepackage{amssymb}
\usepackage{MnSymbol}
\usepackage{rotating}
\usepackage{amssymb}
\usepackage{gastex}
\usepackage[usenames]{color}

\newtheorem{thm}{Theorem}[section]
\newtheorem{cor}[thm]{Corollary}
\newtheorem{lem}[thm]{Lemma}

\newcommand{\abs}[1]{\left\vert#1\right\vert}

\def\abs#1{\ensuremath{\lvert #1\rvert}}
\newcommand{\C}{minimal non-nilpotent}

\begin{document}
\title[Minimal Non-Nilpotent Semigroups]{Finite semigroups that are minimal for not being Malcev nilpotent}

\author{E. Jespers and M.H. Shahzamanian}
\address{Eric Jespers and M.H. Shahzamanian\\ Department of Mathematics,
Vrije Universiteit Brussel, Pleinlaan 2, 1050 Brussel, Belgium}
\email{efjesper@vub.ac.be, m.h.shahzamanian@vub.ac.be}
\address{Current address M.H. Shahzamanian\\
Centro de Matematica Universidade do Porto, Rua do Campo Alegre 687, 4169-007 Porto, Portugal}
\email{ m.h.shahzamanian@fc.up.pt}
\thanks{ 
\\The research of the first author is  partially supported by
Onderzoeksraad of Vrije Universiteit Brussel, Fonds voor
Wetenschappelijk Onderzoek (Belgium). A large part of this work was done while the second author was working at Vrije Universiteit Brussel. The second  author also gratefully acknowledges support by FCT through the Centro de Matematica da Universidade
do Porto (Portugal).}

\begin{abstract}
We give a description of  finite semigroups  $S$ that are minimal
for not being Malcev nilpotent, i.e. every proper subsemigroup and every proper Rees factor semigroup is Malcev nilpotent but $S$ is not. For groups this question was considered by Schmidt.
\end{abstract}

\maketitle

\section{Introduction}\label{pre}

Finite groups $G$ that are minimal for not being nilpotent, i.e. $G$ is not nilpotent but every proper subgroup is nilpotent,   have  been characterized by  
Schmidt in \cite{Schmidt} (see also \cite[Theorem 6.5.7]{Scott} or \cite[Theorem Schmidt-Redei-Iwasawa]{Lausch}). 
For simplicity we call such a group $G$ a  \textit{Schmidt group}. 
It  has   the following  properties:
\begin{enumerate}
\item $\abs{G} =p^aq^b$ for some distinct primes  $p$ and   $q$ and some   $a, b >
0$.
\item  $G$ has a normal Sylow $p$-subgroup and the Sylow
$q$-subgroups are cyclic. 
\item The Frattini subgroups of Sylow
subgroups of $G$ are central in $G$. 
\item  $G$ is two-generated,
i.e. $G=\langle g_{1},g_{2}\rangle$ for some $g_{1},g_{2}\in G$.
\end{enumerate}

It is well known that  nilpotent groups can be defined by using semigroup identities (that is without using inverses) and hence there is a natural notion of nilpotent semigroup.
This was introduced by Malcev (\cite{malcev}),
and independently by Neuman and Taylor (\cite{neu-tay}). For completeness' sake we recall the definition.
For elements $x,y,z_{1},z_{2},\ldots $ in a semigroup $S$  one
recursively defines two sequences
$$\lambda_n=\lambda_{n}(x,y,z_{1},\ldots, z_{n})\quad{\rm and}
\quad \rho_n=\rho_{n}(x,y,z_{1},\ldots, z_{n})$$ by
$$\lambda_{0}=x, \quad \rho_{0}=y$$ and
$$\lambda_{n+1}=\lambda_{n} z_{n+1} \rho_{n}, \quad \rho_{n+1}
=\rho_{n} z_{n+1} \lambda_{n}.$$ A semigroup $S$ is said to be
\textit{nilpotent} (in the sense of Malcev \cite{malcev}, denoted
(MN) in \cite{Riley}) if there exists a positive integer $n$ such
that
$$\lambda_{n}(a,b,c_{1},\ldots, c_{n}) = \rho_{n}(a,b,c_{1},
\ldots, c_{n})$$ for all $a, b$ in $S$ and $ c_{1}, \ldots, c_{n}$
in $S^{1}$ (by $S^{1}$ we denote the smallest monoid containing $S$). The smallest such $n$ is called the nilpotency class
of $S$.
Note that, as in \cite{lal},
the defining condition  to be nilpotent  is a bit stronger than the one required
by Malcev in \cite{malcev}, who requires elements $w_{i}$ in $S$ only.
However the definitions agree on the class of cancellative semigroups.
Furthermore,   it is  shown that  a cancellative semigroup $S$ is nilpotent of class $n$
if and only if $S$ has a  two-sided group of quotients which is nilpotent of class $n$
      (see also \cite{Okninski}).  
Obviously, other examples of nilpotent semigroups are the power nilpotent semigroups,   that is,  semigroups $S$ with zero $\theta$ such that $S^{m}=\{ \theta \}$ for some $m\geq 1$.
In \cite{jespers-okninski99} it is shown that a completely $0$-simple semigroup $S$  over a maximal group $G$ is nilpotent if and only if $G$ is nilpotent and $S$ is an inverse semigroup. 
Of course subsemigroups and Rees factor semigroups of nilpotent semigroups are again nilpotent.
The class of $2$-nilpotent semigroups has been described in \cite{jespers-okninski99}; as for commutative semigroups they have a semilattice decomposition into Archimedean semigroups. 
For more information on this topic  we refer the reader to \cite{jespers-okninski99,jespers-okninski-book,Riley,Jes-shah2}.
In particular, in \cite{Jes-shah2}, we  describe a class of finite
semigroups that are near to being nilpotent, called pseudo
nilpotent semigroups. Roughly said, in these semigroups being  nilpotent  lifts through ideal chains.

In this paper we continue investigating  finite semigroups that are close to being nilpotent. Recall that a proper
Rees factor semigroup of a semigroup $S$ is a Rees factor semigroup $S/I$
with $I$ an ideal of $S$ of cardinality greater than $1$.
Obviously, every  finite semigroup that is not nilpotent has a subsemigroup that is minimal
for not being nilpotent, in the sense that  every proper subsemigroup and every Rees factor semigroup is nilpotent. 
We simply call such a semigroup a \textit{minimal non-nilpotent} semigroup.
The aim of this paper is to describe such semigroups and thus extend Schmidt's investigations to the class of finite semigroups.

The main result (Theorem~\ref{main-theorem})  is a classification (of sorts) of minimal non-nilpotent finite semigroups.  More specifically, it is shown that such a semigroup is either a Schmidt group or one of four types of semigroups which are not groups. These four types of semigroup are each the union of a completely $0$-simple inverse ideal and a $2$-generated subsemigroup or a cyclic group. It is also shown that not every semigroup of these four types is minimal non-nilpotent. The proof of the main theorem utilizes the fact that a minimal non-nilpotent semigroup $S$ which is not a group or a semigroup of left or right zeros has a completely $0$-simple inverse ideal $M$ and $S$ acts on the ${\mathcal R}$-classes of $M$. The different types of orbits of this action are analyzed to provide the classification in Theorem~\ref{main-theorem}.



For standard notations and terminology we refer to \cite{cliford}.
A completely $0$-simple finite semigroup $S$ is isomorphic with a
regular Rees matrix semigroup $\mathcal{M}^{0}(G, n,m;P)$, where
$G$ is a maximal subgroup of $S$, $P$ is the $m\times n$ sandwich
matrix with entries in $G^{\theta}$ and $n$ and $m$ are positive
integers. The nonzero elements of $S$ we denote by $(g;i,j)$, where $g\in
G$, $1\leq i \leq n$ and $1\leq j\leq m$; the zero
element is simply denoted $\theta$. The element  of $P$ on the
$(i,j)$-position we denote by $p_{ij}$. The set of nonzero elements we
denote by $\mathcal{M} (G,n,m;P)$. If all elements of $P$ are
nonzero then this is a semigroup and every completely simple
finite semigroup is of this form. If $P=I_{n}$, the identity
matrix, then $S$ is an inverse semigroup. By what is mentioned
earlier, a completely $0$-simple semigroup
$\mathcal{M}^{0}(G,n,m;P)$ is nilpotent if and only if $n=m$,
$P=I_{n}$ and $G$ is a nilpotent group [\cite{jespers-okninski99}, Lemma 2.1].

The outline of the paper is as follows.  In Section 2 we show that a finite minimal non-nilpotent semigroup is either  a Schmidt group, or a semigroup with $2$ elements of left or right zeros, or $S$ has an ideal  $M$ that is a completely $0$-simple inverse semigroup with nilpotent maximal subgroups. In the latter case we prove that $S$ acts on the ${\mathcal R}$-classes of $M$. 
Next the different types of orbits of this action are analyzed; three cases show up.  In Section 3 we deal with each of these cases separately. As a consequence, we obtain in Section 4 a description of finite minimal non-nilpotent semigroups.


\section{Properties of \C\ Semigroups}\label{non-nil}

We begin by showing that  a finite minimal non-nilpotent semigroup  is either a Schmidt group, a non-commutative semigroup with two elements or it has an ideal that is completely $0$-simple  inverse semigroup with nilpotent maximal subgroups.
The starting point of our investigations is the following necessary and sufficient condition for a finite semigroup not to
be nilpotent  \cite{Jes-shah}.

\begin{lem} \label{finite-nilpotent}
A finite semigroup $S$ is not nilpotent if and only if there exists a positive integer $m$, distinct elements $x, y\in S$ and elements 
  $ w_{1}, w_{2}, \ldots, w_{m}\in S^{1}$ such that $x = \lambda_{m}(x, y, w_{1}, w_{2}, \ldots, w_{m})$, $y = \rho_{m}(x,y, w_{1}, w_{2}, \ldots, w_{m})$.
\end{lem}

Recall that  if $S$ is a semigroup with an ideal $I$ such that both $I$ and $S/I$ are nilpotent semigroups then it does not
follow in general that $S$ is nilpotent. For counter examples we refer the reader to  \cite{jespers-okninski99}. However, if $I^{n}=\{ \theta\}$ (with $\theta$ the zero element of $S$) and $S/I$ is nilpotent then $S$ is a nilpotent semigroup. This easily
follows from the previous lemma.
%
%
%

%


It is easily verified that a finite semigroup of minimal cardinality that is  a \C\ semigroup but is not a  group is the semigroup of right or left zeros with $2$ elements.  
Obviously, these are bands. For convenience and in  order to have  a uniform notation for our main result (Theorem~\ref{main-theorem}) we will denote such   bands  respectively as  $U_1 =\{ e,f\}$, with $ef=f$, $fe=e$, and $U_2 =\{ e',f'\}$ with $e'f'=e'$,  $f'e'=f'$.  
They also can be described as the completely simple semigroups 
   $\mathcal{M}(\{e\},
1, 2; \left(\begin{matrix} 1
\\  1
\end{matrix} \right))$ and $\mathcal{M}(\{e\},
2, 1; (1,1))$.
The following result is a first step towards our classification result.
It turns out that these are precisely the \C\ finite semigroups that are completely simple.

\begin{lem} \label{starter}
Let $S$ be a finite semigroup. If $S$ is \C\ then one of the
following properties hold:
\begin{enumerate}
\item $S$ is  a \C\ group; 
\item $S$ is a semigroup of left or right zeros with 2 elements;
\item $S$ has a proper ideal
which is a completely $0$-simple inverse semigroup with nilpotent maximal subgroups, i.e. $S$ has an ideal isomorphic to $\mathcal{M}^{0}(G, n,n;I_{n})$ where $G$ is a nilpotent group and  $n\geq 2$.
\end{enumerate}
\end{lem}

\begin{proof}
Since $S$ is finite, it has a principal series
$$S= S_1 \supset S_2 \supset \cdots \supset S_{h'} \supset S_{h'+1} = \emptyset .$$ That is, each
$S_i$ is an ideal of $S$ and there is no ideal of $S$ strictly
between $S_i$ and $S_{i+1}$ (for convenience we call the empty set
an ideal of $S$). Each principal factor $S_i / S_{i+1} (1 \leq i
\leq m)$ of $S$ is either completely $0$-simple, completely simple
or null.

Assume $S$ is \C. So, by  Lemma~\ref{finite-nilpotent}, there exist a positive integer $h$, distinct elements $s_1,\, s_2 \in S$ and elements $w_1, w_2, \ldots, w_h\in S^1$ such that
\begin{eqnarray} \label{simple-identity}
s_1 =
\lambda_{h}(s_1, s_2, w_{1}, w_{2}, \ldots, w_{h}) &\mbox{and} & s_2 =
\rho_{h}(s_1, s_2, w_{1}, w_{2}, \ldots, w_{h}).
\end{eqnarray}
Suppose that $s_1 \in S_i \backslash S_{i+1}$. 
Because $S_i$ and $S_{i+1}$ are ideals of $S$, the equalities (\ref{simple-identity}) imply that
   $s_2 \in S_i \backslash S_{i+1}$ and $w_{1}, w_{2}, \ldots, w_{h} \in S^1 \backslash S_{i+1}$. 
Furthermore, one obtains that $S_i / S_{i+1}$ is a completely $0$-simple, say $\mathcal{M}^{0}(G, n,m;P)$, or a
completely simple semigroup, say $\mathcal{M}(G, n,m;P)$. 
Also, since $S$ is \C, $S_{i+1}=\emptyset$ or $S_{i+1}=\{ \theta \}$.

If $n+m=2$ then $S_i \backslash S_{i+1}$ is a group. We denote by $e$ its  identity element. 
Since $s_1,s_2 \in S_i \backslash S_{i+1}$, the sequences $s_1 =
\lambda_{h}(s_1, s_2, w_{1}, w_{2}, \ldots, w_{h})$, $s_2 = \rho_{h}(s_1,
s_2, w_{1}, w_{2}, \ldots, w_{h})$ imply that 
$\lambda_j, \rho_j \in S_{i}\backslash S_{i+1}$ for $1 \leq j \leq h$.
Now, as $e$ is the identity element of $S_{i}\backslash S_{i+1}$,
$\lambda_j w_{j+1} \rho_i=\lambda_j e w_{j+1} \rho_i$ and $\rho_j w_{j+1} \lambda_j=\rho_ie w_{j+1}
 \lambda_j$ for $0 \leq j \leq h-1$. If $S_{i+1} = \{\theta\}$ and $ew_{j+1}= \theta$ for some $0 \leq j \leq h-1$, then $s_1=s_2=\theta$, in  contradiction with $s_1 \neq s_2$.
So, $ew_{j+1} \in S_i \backslash S_{i+1}$ for $0 \leq j \leq h-1$. Consequently
$$s_1 =\lambda_{h}(s_1, s_2, w_{1}, w_{2}, \ldots, w_{h})=\lambda_{h}(s_1, s_2, ew_{1}, ew_{2},
\ldots, ew_{h}) $$ $$\neq s_2 = \rho_{h}(s_1,s_2, w_{1}, w_{2}, \ldots, w_{h})=\rho_{h}(s_1,s_2, ew_{1},
ew_{2}, \ldots, ew_{h}).$$
From
Lemma~\ref{finite-nilpotent} it follows that $S_i \backslash S_{i+1}$ is a group that is not nilpotent.
Hence, since $S$
is \C, $S=S_{i}\backslash S_{i+1}$ and so $S$ is a \C\ group.

Suppose  $n+m>2$. If $S_{i}/S_{i+1}$ is not nilpotent then (by the results mentioned in the introduction ([\cite{jespers-okninski99}, Lemma 2.1]) a column or row of $P$ has two nonzero elements. Without loss of generality, we may suppose this is either the first column or the first row.
If   $p_{i1}$ and
$p_{j1}$ are nonzero with $i \neq j$ then it is  easily verified that the subsemigroup $\langle (p_{i1}^{-1};1,i),\, (p_{j1}^{-1};1,j)\rangle$ is
isomorphic with the \C\ semigroup $U_{1}$. If, on the other hand,  the first row of $P$ contains two nonzero elements, then the semigroup is  isomorphic with the \C\ semigroup $U_{2}$.
So,  $S$ is a semigroup of  right or left zeros with two elements.

The remaining case is $n+m>2$ and $S_{i}/S_{i+1}$ is a nilpotent
semigroup. Again by [\cite{jespers-okninski99}, Lemma 2.1], in this case,
$S_{i}/S_{i+1}=\mathcal{M}^{0}(G, n,n;I_{n})$ with $G$ a nilpotent
group. Since $\abs{S_{i+1}} \leq 1$, the result follows.
%
\end{proof}

In order to obtain a classification, we thus assume throughout the remainder
of this section that $S$ is a finite
\C\ semigroup  that has a proper  ideal $M=\mathcal{M}^{0}(G,n,n;I_{n})$
with $G$ a nilpotent group and $n>1$. 
To further refine our way towards a classification, we introduce an action of
$S$ on the ${\mathcal R}$-classes of $M$, i.e.
we define  a representation (a semigroup homomorphism)
  $$\Gamma : S\longrightarrow \mathcal{T}_{\{1, \ldots, n\} \cup \{\theta\}} ,$$
where $\mathcal{T}$ denotes  the full
transformation semigroup $\mathcal{T}_{\{1, \ldots, n\} \cup
\{\theta\}}$ on the set $\{1, \ldots, n\} \cup \{\theta\}$.
The definition is as follows, for $1\leq i\leq n$ and $s\in S$,
   $$\Gamma(s)(i) = \left\{ \begin{array}{ll}
      i' & \mbox{if} ~s(g;i,j)=(g';i',j)  ~ \mbox{for some} ~ g, g' \in G, \, 1\leq j \leq n\\
       \theta & \mbox{otherwise}\end{array} \right.$$
and
      $$\Gamma (s)(\theta ) =\theta .$$

We call $\Gamma$ a \C\ representation of $S$  and $\Gamma(S)$ a \C\
image of $S$. It is easy to
check that $\Gamma$ is well-defined and that it is a semigroup homomorphism.

Also, for every $s \in S$, we  define a map
\begin{eqnarray} \label{map-psi}
\Psi(s): \{1, \ldots, n\}\cup \{\theta\} & \longrightarrow &G^{\theta}
\end{eqnarray}
as follows
$$\Psi(s)(i)=g  \;\; \;\; \mbox{ if } \Gamma(s)(i) \neq \theta
\mbox{ and } s(1_G;i,j)=(g;\Gamma(s)(i),j)$$ for some  $1\leq j \leq
n,$ otherwise $\Psi(s)(i)=\theta$. It is straightforward to verify
that $\Psi$ is well-defined.

Note that if $\Psi(s)(i)=g$ and $g \in G$ then
$s(h;i,j)=(gh;\Gamma(s)(i),j)$ for every $h \in G$. Also if
$\Psi(st)(i)=g$, $\Psi(t)(i)=g'$ and $\Psi(s)(\Gamma(t)(i))=g''$, then
$g=g''g'$. Hence it follows that 
  $$\Psi (st) =(\Psi (s) \circ \Gamma (t) ) \; \Psi (t).$$

We claim that for $s\in S$ the map $\Gamma (s)$ restricted to the domain $S\backslash \Gamma(s)^{-1}(\theta )$ is injective.
Indeed,   suppose $\Gamma(s)(m_1) = \Gamma(s)(m_2)= m$ for $1
\leq m_1, m_2, m \leq n$. Then there exist  $g,g',h,
h' \in G$, $1 \leq l, l' \leq n$ such that $s(g;m_1,l)=(g';m,l)$ and
$s(h;m_2,l')=(h';m,l')$. Hence
$$(1_G;m,m)s(g;m_1,l)=(g';m,l),\;\;
(1_G;m,m)s(h;m_2,l')=(h';m,l')$$ and thus
$$(1_G;m,m)s= (x;m,m_1)= (x';m,m_2)$$ for some $x, x' \in
G$. It implies that $m_1= m_2$, as required.

It follows that if $\theta \not\in \Gamma(s) (\{ 1, \ldots, n\})$ then $\Gamma (s)$ induces a permutation on $\{ 1, \ldots,
n\}$ and we may write $\Gamma (s)$ in the disjoint cycle notation
(we also write cycles of length one). In the other case, we may
write $\Gamma (s)$ as a product of disjoint cycles of the form
$(i_{1}, i_{2}, \ldots, i_{k})$  or of the form $(i_{1}, i_{2},
\ldots, i_{k}, \theta )$, where $1\leq i_{1}, \ldots, i_{k}\leq n$.
The notation for  the latter cycle means that $\Gamma (s)(i_{j})=i_{j+1}$ for $1\leq
j\leq  k-1$, $\Gamma  (s)(i_{k})=\theta$, $\Gamma  (s)(\theta ) =\theta$ and
there does not exist $1\leq r \leq n$ such that  $\Gamma
(s)(r)=i_{1}$. We also agree that letters $i,j,k$ represent
elements of $\{ 1, \ldots, n\}$, in other words we write
explicitly $\theta$ if the zero appears in a cycle.
Another agreement we make is that we do not write cycles of the form $(i, \theta )$ in
the decomposition of  $\Gamma (s)$  if $\Gamma(s)(i) =\theta$ and
$\Gamma(s)(j) \neq i$ for every $1\leq j \leq n$ (this is the reason for writing cycles of length one).  If $\Gamma(s)(i)
=\theta$ for every $1\leq i \leq n$, then we simply denote $\Gamma
(s)$ as $\theta$.

For convenience, we also introduce the following notation. 
If the cycle $\varepsilon$ appears in the expression of  $\Gamma(s)$ as
product of disjoint cycles then we denote this by $\varepsilon
\subseteq \Gamma(s)$. If $\Gamma(s)(i_1)=i_1', \ldots,
\Gamma(s)(i_m)=i_m'$ then we write  $$[\ldots, i_1, i_1', \ldots,
i_2, i_2', \cdots, \ldots, i_m, i_m', \ldots] \sqsubseteq
\Gamma(s).$$

It can be easily verified that if $g\in G$ and $1 \leq n_1, n_2 \leq n$ with
$n_1 \neq n_2$ then
\begin{eqnarray} \label{elements-of-ideal}
\Gamma((g;n_1,n_2)) = (n_2,n_1,\theta) &\mbox{  and }& \Gamma((g;n_1,n_1)) = (n_1) .
\end{eqnarray}
Further, for $s,t\in S$, $g\in G$ and $1\leq i \leq n$, if
$(\ldots, o, m, k,
\ldots) \subseteq \Gamma(s)$ and $(m_1, \ldots, m_2,$ $ \theta)
\subseteq \Gamma(t)$ then
$$s(g;m,i)=(g';k,i), \;\; t(g;m_2,i)=\theta ,$$
for some $g' \in G$.
Since $s(g; o, i) = (g''; m, i)$ for some $g'' \in G$ we obtain that
$$(g; i,m)s(g; o, i) = (g; i, m)(g''; m, i)= (gg''; i, i)$$ and thus $(g; i,m)s(g;
o, i)\neq \theta$. Hence,   there exists $k \in G$ such that
 $$(g; i,m)s = (k; i, o).$$
We claim that  $$(g;i,m_1)t=\theta .$$
Indeed,
suppose this is not the case. Then $(g;i,m_1)t=(g';i,m_3)$ for
some $g'\in G$ and some $m_{3}$. Hence, $(g;i,m_1)t(1_G;m_3,m_3)
\neq \theta$ and $t(1_G;m_3,m_3) \neq \theta$. So
$\Gamma(t)(m_3)=m_1$, a contradiction.

In the following lemma we analyze the   orbits of this action. 
Three cases show up.

\begin{lem}\label{nilpotency}
Let $S$ be a finite \C\ semigroup. Suppose $S$  has proper  ideal $M =
\mathcal{M}^0(G,n,n;I_{n})$ with  $G$ a nilpotent group and $n\geq 2$. Then, there
exist elements $w_1$ and $w_2$ of $S \backslash M$  such that one of the following properties holds:
\begin{enumerate}
\item[(i)] $(m,l) \subseteq \Gamma(w_1), (m)(l) \subseteq \Gamma(w_2), $
\item[(ii)] $(\ldots, m,l,m', \ldots) \subseteq \Gamma(w_1),  (l)(\ldots, m,
m',\ldots)\subseteq \Gamma(w_2),$
\item[(iii)] $[\ldots, k,m, \ldots, l,k', \ldots]\sqsubseteq \Gamma(w_1), [\ldots, l, m, \ldots, k, k', \ldots] \sqsubseteq \Gamma(w_2)$,
\end{enumerate}
for some pairwise distinct numbers $l, m, m',k$ and $k'$ between $1$
and $n$.
\end{lem}

\begin{proof} Because of Lemma~\ref{finite-nilpotent} there exists
a positive integer $h$, distinct elements $s_1 , s_2\in S$ and
elements $w_{1}, w_{2}, \ldots, w_{h}\in S^{1}$ such that $s_1 =
\lambda_{h}(s_1, s_2 , w_{1},$ $ w_{2}, \ldots, w_{h})$, $s_2 =
\rho_{h}(s_1, s_2, w_{1}, w_{2}, \ldots, w_{h})$. Note that both
$s_{1}$ and $s_{2}$ are nonzero. Because the semigroups $M$ and
$S /M$ are nilpotent, $\{ s_{1} , s_{2} , w_1, $ $\ldots, w_h
\} \cap M \neq \emptyset$ and $\{ s_{1} , s_{2} , w_1,
\ldots, w_h \} \cap (S \backslash M) \neq \emptyset$. It
follows that $s_{1},s_{2}\in M$ and that there exist $1 \leq
n_1,n_2,n_3,n_4 \leq n$ and  $g, g' \in G$ such that $s_1 =(g;
n_1, n_2), \, s_2 = (g'; n_3, n_4)$. Hence
$$[\ldots, n_3,n_2, \ldots, n_1,n_4, \ldots] \sqsubseteq \Gamma(w_1),
[\ldots, n_1, n_2, \ldots, n_3, n_4, \ldots] \sqsubseteq
\Gamma(w_2).$$
Here we agree that we take $w_2=w_1$ in case $h=1$.

If $(n_1,n_2)=(n_3,n_4)$ (for example in the case that $h=1$)
then, there exist $k_{i} \in G$ such that $(k;\alpha,n_2)w_i=
(kk_{i};\alpha,n_1) \text{ for every } k \in G \text{ and } \alpha
\in \{n_1,n_2\}.$ Since  $\lambda_{i-1}= (g_{i-1};n_1,n_2)$,
$\rho_{i-1}= (g'_{i-1};n_1,n_2)$, for some $g_{i-1},g_{i-1}'\in
G$, we  get that   $$\lambda_{i}= (g_{i-1}k_ig'_{i-1};n_1,n_2),
\rho_{i}= (g'_{i-1}k_ig_{i-1};n_1,n_2)$$
 and thus
$$g = \lambda_{h}(g, g', k_{1}, \ldots, k_{h}), g' =
\rho_{h}(g, g',k_{1}, \ldots, k_{h}).$$  Because of Lemma~\ref{finite-nilpotent}, this yields a
contradiction with $G$ being nilpotent. So we have shown that $(n_{1},n_{2})\neq (n_{3},n_{4})$. In particular, we obtain that $h>1$.

We deal with two mutually exclusive cases.

(Case 1) $n_1=n_2=l$. Since $[\ldots, n_1, n_2, \ldots, n_3, n_4,
\ldots]\sqsubseteq \Gamma(w_2)$, it is impossible that $n_3=l, n_4
\neq l$ or $n_3 \neq l,  n_4 = l$. As
$(n_1,n_2)\neq (n_3,n_4)$ we thus obtain that  $n_3 \neq l$ and $n_4 \neq l$. Consequently,
$$[\ldots, m,l, \ldots, l,m, \ldots] \sqsubseteq \Gamma(w_1),~
[\ldots, l, l, \ldots, m, m, \ldots] \sqsubseteq \Gamma(w_2)$$
$$\mbox{or~}[\ldots, m,l, \ldots, l,m', \ldots] \sqsubseteq \Gamma(w_1),~
[\ldots, l, l, \ldots, m, m', \ldots]\sqsubseteq \Gamma(w_2)$$ and
thus $$(m,l) \subseteq  \Gamma(w_1), (l)(m)\subseteq  \Gamma(w_2)$$
$$\mbox{or~} (\ldots, m,l,m', \ldots)\subseteq  \Gamma(w_1),
(l)(\ldots, m, m', \ldots)\subseteq  \Gamma(w_2)$$ for the pairwise
distinct numbers $l, m$ and $m'$.

(Case 2) $n_1 \neq n_2$ and $n_3\neq n_4$ (the latter because otherwise, by symmetry reasons, we are as in Case 1).
We obtain five possible cases:\\
\begin{flushleft} $(1)[n_2=n_3=m]:$\end{flushleft}
$$[\ldots, m,m, \ldots, n_1,n_4, \ldots] \sqsubseteq
\Gamma(w_1),~ [\ldots, n_1, m, \ldots, m, n_4, \ldots]\sqsubseteq 
\Gamma(w_2),$$
$(2)[n_2=n_4=m]:$ $$[\ldots, n_3,m, \ldots, n_1,m, \ldots]\sqsubseteq
\Gamma(w_1),~ [\ldots, n_1, m, \ldots, n_3, m, \ldots]\sqsubseteq
\Gamma(w_2),$$
$(3)[n_1=n_3=m]:$ $$[\ldots, m,n_2, \ldots, m,n_4, \ldots]\sqsubseteq
\Gamma(w_1),~ [\ldots, m, n_2, \ldots, m, n_4, \ldots]\sqsubseteq
\Gamma(w_2),$$
$(4)[n_1=n_4=m]:$ $$[\ldots, n_3,n_2, \ldots, m,m, \ldots]\sqsubseteq
\Gamma(w_1),~ [\ldots, m, n_2, \ldots, n_3, m, \ldots]\sqsubseteq
\Gamma(w_2),$$
$(5):$
$$[\ldots, k,m, \ldots, l,k', \ldots]\sqsubseteq \Gamma(w_1),~
[\ldots, l, m, \ldots, k, k', \ldots]\sqsubseteq \Gamma(w_2)$$ for
some pairwise distinct positive integers  $l, m,k,k'\leq n$.

Cases one and four are as in (ii)  of the
statement of the lemma. Note that cases two and three are not possible, since $(n_1, n_2) \neq (n_3, n_4)$ and the restriction of $\Gamma(w_1)$ to $\{ 1, \ldots, n
\} \backslash  \Gamma(w_1)^{-1}(\theta )$ is an
injective map.
Case five is one of the desired options.

Finally, because of (\ref{elements-of-ideal}) we know how the
elements of $M$ are written as products of disjoint cycles. Hence
it is easily seen that $w_1,\, w_2\in (S\backslash M)$.
\end{proof}

\section{Three types of semigroups}

In this section we deal with each of the cases listed in Lemma~\ref{nilpotency}.
For the first case we obtain the following description.

\begin{lem}\label{nilpotency-case1}
Let $S$ be a finite \C\ semigroup. Suppose   $M =
\mathcal{M}^0(G,n,n;I_{n})$ is a proper ideal, with  $G$ a nilpotent group and $n\geq 2$.
If
there exists $u\in S \backslash M$ such that
$(m,l) \subseteq \Gamma(u)$, 
then
 $$S=\mathcal{M}^0(G,2,2;I_{2}) \cup \langle u \rangle,$$
a disjoint union, and
\begin{enumerate}
\item $\langle u\rangle$ a cyclic group of order $2^{k}$,
\item   $u^{2^{k}}=1$ is the identity of $S$,
\item $\Gamma(u)=(1,2)$ and $\Gamma(1)=(1)(2)$,
\item $G=\langle \Psi(u)(1), \; \Psi(u)(2) \rangle$,
\item  $(\Psi(u)(1)\; \Psi(u)(2))^{2^{k-1}}=1$.
\end{enumerate}
A semigroup $S=\mathcal{M}^0(G,2,2;I_{2}) \cup \langle u \rangle$ that satisfies these five properties is said to be  of  type $U_{3}$.

Furthermore, 
$S= \langle
(g;i,j), u\rangle$ for any (nonzero)  $(g;i,j) \in \mathcal{M}^0(G,2,2;I_{2})$. 
\end{lem}

\begin{proof} Let $1_G$ denote the identity of $G$. Obviously  $(m,l) \subseteq \Gamma(u)$ implies that  $ (m)(l) \subseteq \Gamma(u^2)$. Then
because of (\ref{elements-of-ideal}),  it is easily seen that
\begin{eqnarray}
\Gamma((1_G;m,l))&=&\lambda_2(\Gamma((1_G;m,l)),\Gamma((1_G;l,m)),\Gamma(u^2),\Gamma(u)), \label{not-nilp1}\\
\Gamma((1_G;l,m))&=&\rho_2(\Gamma((1_G;m,l)),\Gamma((1_G;l,m)),\Gamma(u^2),\Gamma(u)).
\label{not-nilp2}
\end{eqnarray}
Hence the semigroup $\langle u, (g;m,l),(g;l,m) \mid
g\in G\rangle$ is not nilpotent by Lemma~\ref{finite-nilpotent}.
Since $S$ is \C, this implies that $$S=\langle u,
(g;m,l),(g;l,m) \mid g\in G\rangle.$$

Let $g\in G$. Since $u(1_G;m,l)=(x;l,l)$ for some $x\in G$,  we
obtain that   $(g;m,l)u(1_G;m,l)=(gx;m,l) \neq \theta$. Hence,
\begin{eqnarray} \label{left-right}
(g;m,l)u&=&(gx;m,m).
\end{eqnarray}
%
%
In particular, $(g;m,l)u,\, u(g;m,l) \in I=\langle
(g;m,l),\, (g;l,m) \mid g\in G\rangle$. 
Note that $I=\mathcal{M}^0(G,2,2;I_{2})$. Hence $I$ is an ideal in
the semigroup $T=\langle u,I\rangle$. Because of
(\ref{not-nilp1}) and (\ref{not-nilp2}) the semigroup $T$ is not
nilpotent. Furthermore, for any $u'\in \langle u\rangle$
one easily sees that $\Gamma (u')$ has at least two fixed points or
contains a transposition in its disjoint cycle decomposition.
Hence, because of (\ref{elements-of-ideal}), $\Gamma (u') \not\in
\Gamma (M)$. Therefore, $\langle u\rangle \cap M
=\emptyset$.

Consequently, we obtain that $S=\langle u\rangle \cup
M$, a disjoint union, $n=2$ and $\Gamma (u)=(m,l)$. 
It is then clear that in (\ref{not-nilp1})
and (\ref{not-nilp2}) one may replace $u$ by $u^{k_1}$, 
with $k_1$ an odd positive integer.
It follows that the
subsemigroup $\langle u^{k_1},M\rangle$ is not nilpotent. Since $S$
is \C\ this implies that $S=\langle u,M\rangle =\langle
u^{k_1},M\rangle$. So $u=u^{r}$ for some positive integer
$r\geq 3$. Let $r$ be the smallest such positive integer. Then
$u^{r-1}$ is an idempotent and $\langle u\rangle$ is a
cyclic group of even order. As $\langle u\rangle =\langle
u^{k_1}\rangle$ for any odd positive integer $k_1$, we get
that $\langle u\rangle$ has order $2^{k}$ for some positive
integer $k$. 

Without loss of generality we may assume that $m=1$, $l=2$. As $(1_G;1,1)u^2$ $=(\Psi(u)(2)\; \Psi(u)(1);1,1)$ we have
$$(1_G;1,1)u^{2^k+1}=((\Psi(u)(2)\; \Psi(u)(1))^{2^{k-1}}\; \Psi(u)(2);1,2).$$
Since $\langle u\rangle$ has order $2^{k}$, $u^{2^k+1}=u$ and thus
$(\Psi(u)(2)\Psi(u)(1))^{2^{k-1}}= 1_G$ and
$(1_G;1,1)u^{2^k}=(1_G;1,1)$. It follows then easily that
$(x;1,1)u^{2^k}=(x;1,1)$ for any $x\in G$. Similarly one obtains
that $u^{2^k}(x;1,1)=(x;1,1)$,
$u^{2^k}(x;2,2)=(x;2,2)u^{2^k}=(x;2,2)$ for any $x\in G$. Hence,
$u^{2^{k}}$ is the identity of the semigroup $S$.

Let $H=\langle \Psi(u)(1), \Psi(u)(2) \rangle$.
From (\ref{not-nilp1}) and (\ref{not-nilp2}) it easily follows
that the subsemigroup $$\mathcal{M}^0(H,2,2;I_{2})\cup \langle u\rangle$$ is not nilpotent.
Then, we obtain that $$\mathcal{M}^0( H, 2,2;I_2)\cup \langle
u\rangle= \mathcal{M}^0( G, 2,2;I_2)\cup \langle
u\rangle.$$
We now show that $G=H$. Suppose the contrary,
then there exists  $g \in G\backslash  H$. Let $\alpha=(g;1,1)$. Clearly, $\alpha \not \in \mathcal{M}^0(H, 2,2;I_2)$ and thus $\alpha \in \langle
u\rangle$. Since $\Gamma(u^{2k})=(1)(2),$ we get $\Gamma(u^{2k+1})=(1,2)$ for $k>1$ and $\Gamma((g;1,1))=(1,1,\theta)$, a contradiction. 
Thus $G=H$ and 
$S$ is a semigroup of type $U_{3}$.

Now suppose that $(g;i,j) \in \mathcal{M}^0(G,2,2;I_{2})$. If $i=j$ then
\begin{eqnarray}
\Gamma((g;i,i))&=&\lambda_2(\Gamma((g;i,i)),\Gamma(u(g;i,i)u),\Gamma(u),\Gamma(u^2)), \\
\Gamma(u(g;i,i)u)&=&\rho_2(\Gamma((g;i,i)),\Gamma(u(g;i,i)u),\Gamma(u),\Gamma(u^2)).
\end{eqnarray}
Hence the semigroup $\langle (g;i,i), u(g;i,i)u\rangle$ is not nilpotent
by Lemma~\ref{finite-nilpotent}. Since $S$ is \C\ this implies that $S=\langle (g;i,j), u\rangle$.

Otherwise if $i \neq j$ we have
\begin{eqnarray}
\Gamma((g;i,j)u)&=&\lambda_2(\Gamma((g;i,j)u),\Gamma(u(g;i,j)),\Gamma(u),\Gamma(u^2)), \\
\Gamma(u(g;i,j))&=&\rho_2(\Gamma((g;i,j)u),\Gamma(u(g;i,j)),\Gamma(u),\Gamma(u^2)).
\end{eqnarray}
Hence the semigroup $\langle (g;i,j)u, u(g;i,j)\rangle$ is not nilpotent
by Lemma~\ref{finite-nilpotent}. Again since $S$ is \C\ this implies that $S=\langle (g;i,j), u\rangle$.

\end{proof}

Note that not every semigroup of type  $U_{3}$ is  \C. Indeed, let   $S=\mathcal{M}^0(G,2,2;I_{2}) \cup \langle u \rangle$, with   $\langle u\rangle$ a cyclic group of order $2$, $u^{2}=1$ is the identity of $S$,
$\Gamma(u)=(1,2)$, $\Gamma(1)=(1)(2)$, $\Psi(u)(1)=g$, $\Psi(u)(2)=g$ and $G$ a cyclic group $\{1,g\}$. The subsemigroup $$\{(1_G;1,1), (1_G;2,2),(g;1,2),(g;2,1), u,1,\theta\}$$ is isomorphic with  $\mathcal{M}^0(\{e\},2,2;I_{2}) \cup \langle u \rangle$. As this is a proper semigroup and it is not nilpotent, the semigroup   $S$ is of type $U_{3}$ but it is not \C.

Semigroups of type $U_{3}$ show up as obstructions in the  description,  given in \cite{Riley},  of the  structure of  linear semigroups satisfying  certain global and local nilpotence conditions. In particular, it is described when finite semigroups are positively Engel. 
Recall that a semigroup $S$ is said to be \textit{positively Engel}, denoted (PE), if for some positive integer $n\geq 2$, $\lambda_{n}(a,b,1,1,c,c^{2},\ldots, c^{n-2}) =\rho_{n} (a,b,1,1,c,c^{2},\ldots, c^{n-2})$ for all $a, b$ in
$S$ and $c\in S^{1}$.
From Corollary 8 in \cite{Riley}  it follows that one of the  obstructions for a finite semigroup $S$ to be (PE) is that  $S$ has an epimorphic image that has the semigroup $\mathcal{F}_{7}$ of type $U_{3}$ as a subsemigroup, where
$\mathcal{F}_{7}=\mathcal{M}^0(\{e\},2,2;I_{2}) \cup \langle u \mid u^{2}=1  \rangle$.

In order to deal with the second and third cases listed in
Lemma~\ref{nilpotency} we  prove the following lemma.

\begin{lem}\label{nilpotency-lemma}
Let $S= \mathcal{M}^0(G,3,3;I_{3}) \cup \langle w_1, w_2\rangle$
be a semigroup that is the union of the ideal  $
M=\mathcal{M}^0(G,3,3;I_{3})$ and the subsemigroup $T=\langle
w_1,w_2\rangle$. Suppose $\Gamma(w_1)=(2,1,3,\theta)$ and $\Gamma
(w_2)= (2,3,\theta ) (1)$.
 Assume $G$ is a nilpotent group, $\theta$ is the zero element of both $M$ and $S$, and suppose
$w_2w_{1}^{2}=w_{1}^{2}w_2=w_{1}^{3}=w_2w_{1}w_2= \theta$.
Then the  following properties hold.
\begin{enumerate}
\item  $S$ is not nilpotent.
\item $T$ is nilpotent.
\item If a subsemigroup $S'$ of $S$ is not nilpotent, then $\langle w_1, w_2\rangle \subseteq
S'$.
\item Every proper Rees factor semigroup of $S$ is nilpotent.
\end{enumerate}
\end{lem}

\begin{proof}
(1) As
$$\Gamma((1_G;1,1))=\lambda_2(\Gamma((1_G;1,1)),\Gamma((1_G;2,3)),\Gamma(w_1),\Gamma(w_2))$$
and
$$
\Gamma((1_G;2,3))=\rho_2(\Gamma((1_G;1,1)),\Gamma((1_G;2,3)),\Gamma(w_1),\Gamma(w_2))$$
we get from Lemma~\ref{finite-nilpotent} that $S$ is not nilpotent.

(2) Clearly $I=T\backslash \langle w_{2}\rangle$ is an ideal of $T$ and $I^{3}=\{ \theta\}$. Obviously $T/I$ is commutative and thus nilpotent.
Hence, $T$ is nilpotent.

(3) Assume $S'$ is a subsemigroup of $S$ that is not nilpotent.
Again by Lemma~\ref{finite-nilpotent}, there exists a positive
integer $p$, distinct elements $t, t'\in S'$ and $t_{1}, t_{2},
\ldots, t_{p}\in S'^1$ such that $t = \lambda_{p}(t, t', t_{1},
t_{2}, \ldots, t_{p})$, $t' = \rho_{p}(t, t', t_{1}, t_{2},
\ldots, t_{p})$. Since $T$ is nilpotent,
$$\{t, t', t_{1}, t_{2}, \ldots, t_{p}\} \cap \mathcal{M}^0(G,3,3;I_{3}) \neq
\emptyset $$ and since $\mathcal{M}^0(G,3,3;I_{3})$ is an ideal of
$S$, $t$ and $t'$ are in $\mathcal{M}^0(G,3,3;I_{3})$. Since $S'$
is not nilpotent and $\mathcal{M}^0(G,3,3;I_{3})$ is nilpotent, we
obtain that at least one of the elements $t_1, \ldots, t_{p}$ is
in $T$. Now, if necessary, replacing $t$ by $\lambda_{i-1}(t, t',
t_{1}, t_{2}, \ldots, t_{i-1})$ and $t'$ by $\rho_{i-1}(t, t',
t_{1}, t_{2}, \ldots, t_{i-1})$, we may assume that $t_1\in T$.

Write $t=(g_1 ;n_1,n_2)$ and $t'=(g_2 ;n_3,n_4)$,  for some $1\leq
n_1,n_2,n_3,n_4 \leq 3$ and $g_1 , g_2  \in G$.

 Consider the following subsets of $T$: $A=\{w_1\}$,  $B=\{w_{2}\}$,
$C=\{w_{2}^{n} \mid n \in \mathbb{N}, n > 1\}$, $D=\{w_1 w_{2}^{n}
w_1 \mid n \in \mathbb{N}\}$,  $E=\{w_1 w_{2}^n \mid n \in
\mathbb{N}, n > 0\}$, $F=\{w_{2}^{n} w_1 \mid n \in \mathbb{N}, n
> 0\}$ and $Z=\{T^1w_{1}^{2}T^1, T^1w_{2} x_{1} w_{2} T^1 \}
\backslash \{w_{1}^{2}\}.$
By determining the images of these sets under the mapping $\Gamma$ one
sees that these sets form a partition of $T$. Since
$w_{2}w_{1}^{2}=w_{1}^{2}w_{2}=w_{1}^{3}=w_{2}w_{1}w_{2}= \theta$
we have that $Z=\{\theta\}$. Hence $t_1 \not \in Z$.

If $t_1 \in C$
then $n_1=n_2=n_3=n_4=1$ (because, for every $a \in C$ we have $
\Gamma(a)=(1)$) and thus $g_{1} = \lambda_{p}(g_{1},g_{2}, x_{1},
\ldots, x_{p})$ and $g_{2}=\rho_{p}(g_{1},g_{2},x_{1}, \ldots,
x_{p})$ for some $x_{1},\ldots, x_{p}\in G$, in contradiction
with $G$ being nilpotent. If $t_1 \in D$ then $n_1=n_3=2$ and
$n_2=n_4=3$ (because for every $a \in D$ we have
$\Gamma(a)=(2,3,\theta)$); this again yields a contradiction with
$G$ being nilpotent. Similarly $t_1 \not \in E, F$. Now suppose
that $t_1\in A$, i.e. $t_{1}=w_1$. Since
$\Gamma(w_1)=(2,1,3,\theta)$, $t=\lambda_{p}(t,t',t_{1},t_{2},\ldots, t_{p})\neq \theta$ and $t'=\rho_{p}(t,t',t_{1},\ldots, t_{p})\neq \theta$ we get that
$\{n_1,n_3\} \subseteq \{1,2\}$.
As $tw_1t'\neq \theta$ and $t'w_1t\neq \theta$ we obtain that $n_2=\Gamma (w_{1})(n_3)$ and $n_4=\Gamma (w_{1})(n_1)$.
Hence,
if
$n_1=n_3$, then $n_2=n_4$,   again yielding a contradiction with $G$ being
nilpotent. So, $n_1\neq n_3$.
If $n_1=1$ then $n_3=2$, $n_4=3$, $n_2=1$ and thus then $\{(n_1,n_2),(n_3,n_4)\}=\{(1,1),(2,3)\}$. Similarly, we also get the latter if  $n_1=2$. Note that in this case $p>1$. It then can be easily verified that $t_2=w_2$
and thus $T \subseteq S'$, as desired. Similarly if $t_1=w_2$,
then $T \subseteq S'$.

(4) Since $M$ is a completely $0$-simple semigroup and because $T$ is
nilpotent, it is clear that every proper Rees factor of $S$ is
nilpotent. 
\end{proof}

We now are in a position to obtain a description of finite \C\ semigroups that are not of type $U_3$
and that have a proper ideal that is a completely $0$-simple inverse semigroup.

\begin{lem}\label{nilpotency-case2}
Let $S$ be a finite \C\ semigroup with a proper ideal $M =
\mathcal{M}^0(G,n,n;I_{n})$, $G$ a nilpotent group and $n\geq 2$.
Suppose $S$ is not of type $U_{3}$, i.e. for $x \in S\backslash M$ there do not exist distinct numbers $l_1$ and $l_2$ between $1$ and $n$ such that $(l_1,l_2)\subseteq \Gamma(x)$.
Then $S$ is a semigroup of one of the following two types.
\begin{enumerate}
\item $S= \mathcal{M}^0(G,3,3;I_{3}) \cup \langle x_1,
x_2\rangle$, with 
$ \mathcal{M}^0(G,3,3;I_{3})$ a proper ideal of $S$,
$\Gamma(x_1)=(2,1,3,\theta)$, $\Gamma(x_2)=(2,3,\theta)(1)$,
$x_{2}x_{1}^{2}=x_{1}^{2}x_{2}=x_{1}^{3}=x_{2}x_{1}x_{2}= \theta$
(the zero element of $S$), $$G=\langle \Psi(x_1)(1), \,
\Psi(x_1)(2),\,  \Psi(x_2)(1),\,  \Psi(x_2)(2) \rangle .$$
Such a semigroup is said to be of  type $U_{4}$.




\item $S= \mathcal{M}^0(G,n,n;I_{n}) \cup \langle v_1,
v_2\rangle$, with
$ \mathcal{M}^0(G,n,n;I_{n})$ a proper ideal of $S$,
$$[\ldots, k_1, k_2, \ldots, k_3,k_4, \ldots]\sqsubseteq \Gamma(v_1), [\ldots, k_1, k_4, \ldots, k_3, k_2, \ldots] \sqsubseteq \Gamma(v_2)$$
for pairwise distinct numbers $k_1, k_2,k_3$ and $k_4$ between $1$
and $n$, $G=\langle \Psi(v_1)(1),  \ldots, \Psi(v_1)(n),\, \Psi(v_2)(1),  \ldots, \Psi(v_2)(n), \theta
\rangle \backslash \{\theta\}$
and there do not exist pairwise distinct numbers $o_1, o_2$
and $o_3$ between $1$ and $n$ such that $(o_2,o_1,o_3,\theta)\subseteq \Gamma(y)$, $(o_2,o_3,\theta)(o_1)\subseteq \Gamma(z) $
 for some $y,z\in \langle v_1, v_2 \rangle$.
Such a semigroup is said to be of  type $U_{5}$.
\end{enumerate}
Furthermore, if $S$ of type $U_{4}$ then 
 $S=\langle (g;1,1), (g';2,3), x_1,\, x_2 \rangle$ 
and if $S$ is of type $U_{5}$ then 
 $S=\langle (g;k_1, k_4), (g';k_3,
k_2), v_1, v_2 \rangle $, 
for every $g, g' \in G$.
\end{lem}

\begin{proof}
%
For clarity we give a brief outline of the structure of the proof. By assumption $S$ is not of type $U_{3}$ and hence it follows from Lemma~\ref{nilpotency} that we have two cases to deal with and this is done in three  parts. In part (1)  we deal with a special case of part (ii) of Lemma~\ref{nilpotency}.  In the remainder of the proof we then assume that
we are not in this special case.  In part (2) we  deal with all cases occurring in part (iii) of Lemma~\ref{nilpotency} as well as in part (ii) of Lemma~\ref{nilpotency}, the latter  provided some extra condition is satisfied. Finally in part (3), we show that if this extra assumption is not satisfied  then $S$ has to be of   type $U_5$.

Part (1).
 We begin the proof with handling a special case stated in  part (ii) of Lemma~\ref{nilpotency}.  Suppose that there exist 
elements $x_1, x_2\in S$ such that 
     $$(o_2,o_1,o_3,\theta) \subseteq \Gamma(x_{1}) \mbox{ and } (o_2,o_3,\theta)(o_1)\subseteq \Gamma(x_{2}),$$
with $o_1,\, o_2,\, o_3$ positive integers between $1$ and $n$. 
It is then readily verified that $\mathcal{M}^0(G,\{o_1,o_2,o_3\},\{o_1,o_2,o_3\};I_{\{o_1,o_2,o_3\}}$ $)$ is
an ideal in the subsemigroup
$\mathcal{M}^0(G,\{o_1,o_2,o_3\},\{o_1,o_2,o_3\};I_{\{o_1,o_2,o_3\}}$ $) \cup
\langle x_1, x_2 \rangle$ of $S$. Because
\begin{eqnarray}
    ~\Gamma((1_G;o_1,o_1))&=&
         \lambda_2(\Gamma((1_G;o_1,o_1)),\Gamma((1_G;o_2,o_3)),\Gamma(x_1),\Gamma(x_2))          
          \label{eq3u3} , \\
   ~\Gamma((1_G;o_2,o_3))&=&
         \rho_2(\Gamma((1_G;o_1,o_1)),\Gamma((1_G;o_2,o_3)),\Gamma(x_1),\Gamma(x_2)) 
         \label{eq4u3}
\end{eqnarray}
we get that  the semigroup
$\mathcal{M}^0(G,\{o_1,o_2,o_3\},\{o_1,o_2,o_3\};I_{\{o_1,o_2,o_3\}}$ $) \cup
\langle x_1, x_2 \rangle$ is not nilpotent and
thus, as $S$ is \C, $S = \mathcal{M}^0(G,3,3;I_3) \cup \langle
 x_1, x_2 \rangle$.

We now prove that $S$ is a semigroup of type $U_4$. We do so by showing that 
 $x_{1}$ and $x_{2}$ satisfy conditions listed in part (1) of the statement of the lemma.
Let $T=\langle x_{1},x_{2}\rangle$.
Consider the
following subsets of $T$:  $A= \{x_1\}$, $B=\{x_{2}\}$, $C=\{x_{2}^{n}
\mid n \in \mathbb{N}, n > 1\}$, $D=\{x_1 x_{2}^{n} x_1 \mid n \in
\mathbb{N}\}$, $E=\{x_1 x_{2}^n \mid n \in \mathbb{N}, n > 0\}$,
$F=\{x_{2}^{n} x_1 \mid n \in \mathbb{N}, n > 0\}$ and
$Z=\{T^1x_{1}^{2}T^1, T^1x_{2} x_{1} x_{2} T^1 \} \backslash
\{x_{1}^{2}\}$. By determining the images of these sets under the
mapping $\Gamma$ one sees that these sets form a partition of $T$.
Since $S'=\{T^1x_{1}^2T^1, T^1x_{2} x_{1}x_{2} T^1 \} \backslash
\{x_{1}^2\} \cup \{\theta\}$ is an ideal of $S$, it easily follows
from (\ref{eq3u3}) and (\ref{eq4u3}) that $S / S'$ is
non-nilpotent. Hence, $S'=\{\theta\}$ and thus  $Z=\{\theta\}$.
So,   $x_{1}^3= x_{1}^{2}x_{2}=x_{2}x_{1}^{2}=x_{2}x_{1}x_{2}=\theta$, as desired.
From Lemma~\ref{nilpotency-lemma}(2) we know that the
subsemigroup $T=\langle x_{1},x_{2}\rangle $ is nilpotent.

Let $H=\langle \Psi(x_1)(o_1), \Psi(x_1)(o_2),
\Psi(x_2)(o_1), \Psi(x_2)(o_2) \rangle$. 
For simplicity, and without loss of generality, we may assume that 
 $o_1=1$, $o_2=2$ and $o_3=3$. 
From (\ref{eq3u3}) and (\ref{eq4u3}) it can be easily verified that the
subsemigroup
$$\mathcal{M}^0(\langle \Psi(x_1)(1), \Psi(x_1)(2),
\Psi(x_2)(1), \Psi(x_2)(2) \rangle,3,3;I_{3})\cup \langle
x_{1},x_{2}\rangle$$ is not nilpotent.
Hence $$\mathcal{M}^0( H, 3,3;I_{3})\cup \langle
x_{1},x_{2}\rangle= \mathcal{M}^0( G, 3,3;I_{3})\cup \langle
x_{1},x_{2}\rangle.$$
We now show that $G=H$. Suppose the contrary and let
$g \in G\backslash  H$. 
Let $\alpha=(g;1,1)$. Clearly, $\alpha \not \in \mathcal{M}^0(H, 3,3;I_{3})$ and thus $\alpha \in \langle
x_{1},x_{2}\rangle$. Since $(g;1,1)(1_G;1,1) \neq \theta$, we get that  $\Gamma(\alpha)(1), \Psi(\alpha)(1) \neq \theta$ and $(g;1,1)=(g;1,1)(1_G;1,1)=\alpha(1_G;1,1)=(\Psi(\alpha)(1);\Gamma(\alpha)(1),1)$ and thus $g = \Psi(\alpha)(1)$. This contradicts with  $g \not \in H$.
 So, indeed, $G=H$.
%
%
Hence we have shown that indeed $S$ is a semigroup of type $U_{4}$.
To prove the last part of the statement of the lemma for this semigroup, let $g,g'\in G$.
Since 
$$\Gamma((g;1,1))=\lambda_2(\Gamma((g;1,1)),\Gamma((g';2,3)),\Gamma(x_1),\Gamma(x_2))$$
and
$$
\Gamma((g';2,3))=\rho_2(\Gamma((g;1,1)),\Gamma((g';2,3)),\Gamma(x_1),\Gamma(x_2)),$$
we get that the subsemigroup $\langle (g;1,1), (g';2,3), x_1, x_2
\rangle$ is not nilpotent. 
Since $S$ is \C\ it follows that  $S= \langle (g;1,1), (g';2, 3), x_1,$ $ x_2 \rangle$.

Part (2).
In the remainder of the proof we assume that there do not exist pairwise distinct numbers $o_1, o_2$ and $o_3$  between $1$ and $n$ such that $(o_2,o_1,o_3,\theta)\subseteq \Gamma(y)$, $(o_2,o_3,\theta)(o_1)\subseteq \Gamma(z) \mbox{ for some } y,z\in S$.
Because, by assumption $S$ is not of type $U_{3}$, it follows from Lemma~\ref{nilpotency} that $S\backslash  M$ contains elements $w_1$ and $w_2$ such that
  $$(\ldots, m,l,m', \ldots) \subseteq \Gamma(w_1),  (l)(\ldots, m, m',\ldots)\subseteq \Gamma(w_2),$$
or it contains elements $v_1$ and $v_2$ such that 
  $$ [\ldots, k,m, \ldots, l,k', \ldots]\sqsubseteq \Gamma(v_1), [\ldots, l, m, \ldots, k, k', \ldots] \sqsubseteq \Gamma(v_2),$$
for some pairwise distinct numbers $l, m, m',k$ and $k'$ between $1$
and $n$.
In the former case, without 
loss of generality, we may assume that
$m=1$, $l=3$ and $m'=2$.

Assume  the former case holds, i.e. 
    $$(\ldots, 1,3,2, \ldots ) \subseteq \Gamma (w_{1}) \mbox{ and }
    (3) \, (\ldots, 1,2, \ldots ) \subseteq \Gamma (w_{2}) ,$$ 
and also suppose that 
  $$\Gamma(w_1)(2)=r \neq \theta .$$
Then, $[\ldots, 1,2, \ldots, 3,r,
\ldots]\sqsubseteq \Gamma(w_{1}^2), [\ldots, 1, r, \ldots, 3, 2,
\ldots] \sqsubseteq \Gamma(w_1w_2)$.
Hence,
 $v_{1}=w_{1}^{2}$ and $v_{2}=w_{1}w_{2}$ are  elements as in the latter case.
For such elements $\Gamma(v_1)=[\ldots, 1,2, \ldots, 3,r,
\ldots]$ and $\Gamma(v_2)=[\ldots, 1, r, \ldots, 3, 2,$ $
\ldots]$ 
we get that

$$\Gamma((1_G;1,r))=\lambda_2(\Gamma ((1_G;1,r)),\Gamma((1_G;3,2)),\Gamma(v_1),\Gamma(v_2))$$
and
$$
\Gamma((1_G;3,2))=\rho_2(\Gamma((1_G;1,r)),\Gamma((1_G;3,2)),\Gamma(v_1),\Gamma(v_2)).$$
It follows  that
the subsemigroup $\mathcal{M}^0(G,n,$ $n;I_n) \cup \langle v_{1},
v_{2} \rangle$ is not nilpotent. 
So,  because $S$ is \C,  $S
= \mathcal{M}^0(G,n,n;I_n) \cup \langle  v_1, v_2 \rangle$.
Furthermore,
 it is easily verified that
$$\Gamma((g;1, r))=\lambda_2(\Gamma((g;1, r)),\Gamma((g';3, 2)),\Gamma(v_1),\Gamma(v_2))$$
and
$$
\Gamma((g';3,2))=\rho_2(\Gamma((g;1,
r)),\Gamma((g';3,2)),\Gamma(v_1),\Gamma(v_2)).$$ Hence, for any
$g,g'\in G$,  the subsemigroup $\langle (g;1, r),
(g';3,2), v_1, v_2\rangle$ is not nilpotent. Therefore,
$S=\langle (g;1, r), (g';3,2), v_1, v_2 \rangle$ for
every $g, g' \in G$.
We now show that such a semigroup $S$ is of type $U_5$.
Let $$H_{1}=\langle \Psi(v_1)(1),  \ldots,
\Psi(v_1)(n), \Psi(v_2)(1),  \ldots, \Psi(v_2)(n), {\theta} \rangle.$$
Note that $H=H_{1}\backslash \{ \theta \}$ is a subgroup  of  the maximal subgroup $G$ defining $M$.
Since  \begin{eqnarray*}
\Gamma((1_G;3,2))&=&\lambda_2(\Gamma((1_G;3,2)),\Gamma((1_G;1,r)),\Gamma(v_1),\Gamma(v_2)), \label{v1}\\
\Gamma((1_G;1,r))&=&\rho_2(\Gamma((1_G;3,2)),\Gamma((1_G;1,r)),\Gamma(v_1),\Gamma(v_2)) \label{v2}
\end{eqnarray*}
we get that
 $\mathcal{M}^0( H, n,n;I_{n})\cup \langle
v_{1},v_{2}\rangle$ is not nilpotent.
Because, by assumption,  $S$ is \C, we obtain that $$\mathcal{M}^0( H, n,n;I_{n})\cup \langle
v_{1},v_{2}\rangle= \mathcal{M}^0( G, n,n;I_{n})\cup \langle
v_{1},v_{2}\rangle.$$
We now show that $G=H$. Suppose the contrary and let 
$g \in G\backslash  H$. Let $\alpha=(g;1,1)$. Clearly, $\alpha \not \in \mathcal{M}^0(H, n,n;I_{n})$ and thus $\alpha \in \langle
v_{1},v_{2}\rangle$. Since $(g;1,1)(1_G;1,1) \neq \theta$ we get that  
$\Gamma(\alpha)(1) \neq \theta$, $\Psi(\alpha)(1) \neq \theta$ and 
 $$(g;1,1)=(g;1,1)(1_G;1,1)=\alpha(1_G;1,1)=(\Psi(\alpha)(1);\Gamma(\alpha)(1),1)$$
Thus $g = \Psi(\alpha)(1)$. This contradicts with  $g \not \in H$.
 So, indeed, $G=H$.

Part (3). We are left to deal with the case that
 $S\backslash  M$ contains elements $w_1$ and $w_2$ such that
  $$(\ldots, 1,3,2, \ldots) \subseteq \Gamma(w_1),  (3)(\ldots, 1, 2,\ldots)\subseteq \Gamma(w_2),$$
and  
 $$\Gamma(w_1)(2)=\theta .$$
We show that if $S$ is not of type $U_5$ then this can not occur.
First notice that  we also may assume that there
does not exist a positive integer $r'$, with $1\leq r' \leq n$, such that
$\Gamma (w_{1})(r')=1$. Indeed, suppose the contrary and let $r'$ be
such a positive integer.  Then $[\ldots, 1,2, \ldots, r',3,
\ldots]\sqsubseteq \Gamma(w_{1}^2), [\ldots, 1, 3, \ldots, r', 2,
\ldots] \sqsubseteq \Gamma(w_2w_1)$ and thus $S$ is a semigroup of
type $U_{5}$. This proves the claim.
Thus $$(1,3,2,\theta)
\subseteq \Gamma(w_1) .$$
Next we claim that the cycle  $(\ldots, 1, 2, \ldots)$ in
$\Gamma(w_2)$ ends in $\theta$. Indeed, assume the contrary. That is, this cycle ends in a positive integer.  Let $n_{2}$ denote the length of this cycle. Then, $(1,3) \subseteq \Gamma(w_2^{n_2-1}w_1), (1)(3)\subseteq
\Gamma(w_2^{n_2})$. However, this is excluded by assumption.
This proves the claim and thus there
exist positive integers
$k,k'$ and $n'$ such that $\Gamma(w_2)(k)=k'\neq \theta$,
$\Gamma(w_2)(k')= \theta$ and $\Gamma(w_2^{n'})(1)=k'$. So 
 $$(3)\, (
\ldots, 1,2, \ldots, k,k',\theta ) \subseteq \Gamma (w_{2}) .$$
If $\Gamma(w_2^{n'-1}w_1)(k')=k''\neq \theta$ then $[\ldots, 1,k',
\ldots, 3,k'', \ldots]\sqsubseteq \Gamma((w_2^{n'-1}w_1)^2)$,
$[\ldots, 1, k'', \ldots, 3, k', \ldots] \sqsubseteq
\Gamma(w_2^{n'-1}w_1w_2^{n'})$. Since $\Gamma(w_2)(3)=3,
\Gamma(w_2)(1)=2$ and $\Gamma(w_2)(k')=\theta$, it is clear that
$1,k'$ and $3$ are pairwise distinct. As
$\Gamma(w_2^{n'-1}w_1)(k')=k''\neq \theta$, we get that $\Gamma(w_1)(k')=\alpha
\neq \theta$. As $(1,3,2,\theta) \subseteq \Gamma(w_1)$ and because
$1,k',3$ are pairwise distinct positive integers, we get that $\alpha \not
\in \{1,2,3\}$.  We claim that $1,k',3$ and $k''$ are pairwise
distinct and thus that $S$ is a semigroup of type $U_{5}$, yielding a contradiction. 
Indeed, for otherwise, $k''=3$ or $k''=k'$ or $k''=1$. The former is
excluded as it implies that  $\alpha=  3$. If $k''=k'$, then $\alpha=2$, a contradiction. If
$k''=1$ then $(\alpha,3)\subseteq
\Gamma(w_1w_2^{n'-1}w_1w_2^{2n'-1})$ and thus $S$ is of type $U_3$, again a
contradiction. This proves the claim. 
Finally, if $\Gamma(w_2^{n'-1}w_1)(k')=\theta$, we get that $(1,3,k',
\theta) \subseteq \Gamma(w_2^{n'-1}w_1)$. Clearly, $(3)(1,
k',\theta) \subseteq \Gamma(w_2^{n'})$. However, this  contradicts with our assumptions.
This final contradiction shows that indeed this considered case does not occur.
\end{proof}



\section{Main result and examples}\label{main Theorem}

We are now in a position to state and prove the main result.

\begin{thm} \label{main-theorem}
Let $S$ be a finite \C\ semigroup. Then $S$ is either a Schmidt
group or a semigroup of type $U_{1}$, $U_{2}$, $U_{3}$, $U_4$ or $U_{5}$.
In particular, the semigroups $S$ of type $U_m$, with $3\leq m \leq 5$,  are
generated by four elements, they have a two-generated subsemigroup
$T$ and an ideal $M=\mathcal{M}^{0}(G,n,n;I_{n})$ (with $G$ a
nilpotent group) such that
 $$S=M\cup T,$$  and there exists a representation

  $$\Gamma : S\longrightarrow \mathcal{T}_{\{1, \ldots, n\} \cup \{\theta\}} ,$$ 
such  that, for all $s\in S$,
\begin{enumerate}
\item $\Gamma (s) (\theta )=\theta$,
\item $\Gamma (s)$ is injective when
restricted to $\{ 1 ,\ldots, n\} \backslash \Gamma (s)^{-1}(\theta
)$,
\item
$\abs{\Gamma^{-1}(\theta )}=1$,
\item if $T$ has a zero element, say $\theta_{T}$, then $\theta_{T}=\theta$ (the zero of $S$).
\end{enumerate}

Furthermore, $\Gamma (S)$ also is a \C\ semigroup.
\end{thm}
\begin{proof}

All parts, except the items (3) and (4), follow  at once from Lemma~\ref{starter},
Lemma~\ref{nilpotency}, Lemma~\ref{nilpotency-case1} and
Lemma~\ref{nilpotency-case2}.
Part (3)
follows from the fact that $\Gamma^{-1}(\theta )$ is an ideal of $S$
and $S/\Gamma^{-1}(\theta )$ is not nilpotent.

To prove part (4), assume $T$ has a zero element, say $\theta_{T}$. We prove by contradiction that $\theta_{T}=\theta$. So suppose that $\theta_T\neq \theta$. Then, by part (3),
$\Gamma(\theta_T)\neq \theta$. Hence,  there exists $i$ between $1$ and $n$ such that $\Gamma(\theta_T)(i)\neq \theta$.
Now let $t \in T$. We have $$\theta_T(1_G;i,i)=(\Psi(\theta_T)(i) ; \Gamma(\theta_T)(i),i)$$
and
\begin{eqnarray*}
\theta_Tt(1_G;i,i)&=&\theta_T(\Psi(t)(i);\Gamma(t)(i),i)\\
 &=&(\Psi(\theta_T)(\Gamma(t)(i))\Psi(t)(i);\Gamma(\theta_T)(\Gamma(t)(i)),i).
 \end{eqnarray*}
Because $\theta_T t= \theta_T$ we obtain that  $\Gamma(\theta_T)(\Gamma(t)(i))=\Gamma(\theta_T)(i)$. Now as $\Gamma(\theta_T)(i) \neq \theta$, $\Gamma(t)(i)=i$. Therefore, for every $t\in T$, we have $(i) \subseteq \Gamma(t)$.  Because $\Gamma((g;\alpha,\beta))=(\beta,\alpha,\theta)$ for every $(g;\alpha,\beta) \in M$, it follows that  $M \cap T = \emptyset$.
Let $M'=\mathcal{M}^{0}(G,\{1,\ldots,i-1, i+1, \ldots,n\},\{1,\ldots,i-1, i+1, \ldots,n\};I_{n-1})$. Since $(i) \subseteq \Gamma(t)$ and $\Gamma(t)$ restricted to  $\{ 1, \ldots,
n\} \backslash \Gamma(t)^{-1}(\theta)$ is injective for every $t \in T$, we get that $M'T, TM' \subseteq M'$
and thus $M' \cup T$ is a subsemigroup of $S$.
Lemma~\ref{nilpotency}  implies that there
exist elements $w_1$ and $w_2$ of $T$  such that
$$(m,l) \subseteq \Gamma(w_1), (m)(l) \subseteq \Gamma(w_2) $$
$$\mbox{or~} (\ldots, m,l,m', \ldots) \subseteq \Gamma(w_1),  (l)(\ldots, m,
m',\ldots)\subseteq \Gamma(w_2)$$
$$\mbox{or~} [\ldots, k,m, \ldots, l,k', \ldots]\sqsubseteq \Gamma(w_1), [\ldots, l, m, \ldots, k, k', \ldots] \sqsubseteq \Gamma(w_2)$$
for  pairwise distinct numbers $l, m, m',k$ and $k'$ between $1$
and $n$.
As $\Gamma(w_1)(o)$ $\neq o$ for $o \in \{l, m, m',k,k'\}$, $i \not \in \{l, m, m',k,k'\}$. Hence $M' \cup T$ is not nilpotent. As $T \cap M = \emptyset$ and $M \neq M'$, $S \neq M' \cup T$ and this is in contradiction with $S$ being \C.
\end{proof}

The theorem shows that finite \C\ semigroups that are not a group belong to five classes.
In order to get a complete classification, a remaining problem is to determine which semigroups in these classes are actually \C.  In particular, one has to determine when precisely a union $M\cup T$ of an inverse semigroup $M=\mathcal{M}^{0}(G,n,n;I_{n})$ (with $G$ a
nilpotent group) and a two-generated semigroup $T$ is \C. 
One might expect that the easiest case to deal with is when $M$ and $T$ are $\theta$-disjoint, i.e.
the only possible joint element is the zero element $\theta$. In Corollary~\ref{main-cor} we show  that every finite \C\ semigroup  which is of type $U_3$, $U_4$ or $U_5$ is an epimorphic image of such a semigroup.
However, not every semigroup of type $U_4$ or $U_5$ that is a $\theta$-disjoint union of $M$ and $T$ is \C.
Next we give  examples of  \C\ semigroups of type  $U_5$  for which the maximal subgroups of $M$ are not trivial. We finish by constructing an infinite  class of \C\ semigroups of type $U_5$ with $n\geq 5$ and $G$ the trivial group.

Note that in general the subsemigroups $T$ and $M$  of a \C\
semigroup $U_{m}$ (listed in Theorem~\ref{main-theorem}) are not
$\theta$-disjoint ($\theta$-disjoint means that if there is a common element then it is $\theta$).
We now show that $U_{m}$ (with $3\leq n \leq 5$) is an epimorphic image of a
semigroup built on $\theta$-disjoint semigroups.

Let $T$ be a semigroup with a zero $\theta_T$ and let $M$ be a nilpotent regular Rees matrix semigroup $\mathcal{M}^0(G,n,n;I_n)$.
Let $\Gamma$ be a
representation of $T$ to the full transformation semigroup
$\mathcal{T}_{\{ 1, \ldots, n\} \cup \{\theta\}}$ such that for
every $t\in T$, $\Gamma (t) (\theta ) =\theta$, $\abs{\Gamma^{-1}(\theta
)}\leq 1$ (as agreed before, by $\theta$ we also denote the constant
map onto $\theta$),
 $\Gamma(t)$ restricted to  $\{ 1, \ldots, n\} \backslash \Gamma(t)^{-1}(\theta)$ is injective and $\Gamma (\theta_T)  =\theta$.
Further, for every
$t\in T$, let
$$\Psi(t): \{ 1, \ldots, n\} \cup
\{\theta\}\rightarrow G^{\theta}$$ be a map (as considered in
(\ref{map-psi})) such that $\Psi(t)(i)\neq \theta$ if and only if
$\Gamma(t)(i) \neq \theta$ and
$\Psi (t_{1}t_{2}) =(\Psi (t_{1}) \circ \Gamma (t_{2})) \, \Psi (t_{2})$ for every $t_{1},t_{2}\in T$.

We define a semigroup denoted by
$$S=\mathcal{M}^{0}(G,n,n;I_{n}) \cup_{\Psi}^{\Gamma} T.$$
As sets this is the $\theta$-disjoint union of $\mathcal{M}^{0}(G,n,n;I_{n})$ and $T$ (i.e. the disjoint union with the zeros identified).
The multiplication is such
that $T$ and $M$ are subsemigroups, 
$$t \, (g;i,j) = \left\{ \begin{array}{ll}
                              (\Psi (t)(i) g; \Gamma (t)(i),j) & \mbox{ if } \Gamma (t)(i) \neq \theta\\
                              \theta & \mbox{ otherwise}
                          \end{array} \right.
                          $$
and
$$ (g;i,j) t = \left\{ \begin{array}{ll}
    (g\Psi(t)(j');i,j') & \mbox{  if }  \Gamma (t)(j')=j\\
    \theta &\mbox{ otherwise }
    \end{array} \right.
    $$
It can be easily verified that $S$ is associative.

Note that if $G=\{e\}$, then $\Psi(t)(i)=e$ if and only if $\Gamma(t)(i) \neq \theta$. In this case we denote $\Psi$ simply as  $id$.

It follows from
Theorem~\ref{main-theorem} (and its proof)
that the \C\ semigroup S of type $U_{m}$ (with $3\leq n\leq 5$) is an
epimorphic image of a semigroup of the type
$\mathcal{M}^{0}(G,n,n;I_{n}) \cup_{\Psi}^{\Gamma} T$, with $G$ a
nilpotent group  and $T$ a two-generated nilpotent semigroup with a zero.
%

\begin{cor} \label{main-cor} Every finite \C\ semigroup $S$ is an  epimorphic image of one of the following semigroups:
\begin{enumerate}
\item a Schmidt group,

\item $U_{1}=\{ e,f\}$ with $e^{2}=e$, $f^{2}=f$, $ef=f$ and
$fe=e$,

\item $U_{2}=\{ e,f\}$ with $e^{2}=e$, $f^{2}=f$, $ef=e$ and
$fe=f$,

\item $\mathcal{M}^{0}(G,2,2;I_{2}) \cup_{\Psi}^{\Gamma} T$ such that $T=\langle
u \rangle \cup \{\theta\}$ with $\theta$ the zero of $S$, $u^{2^{k}}=1$ the identity of $T \backslash \{\theta\}$ (and of $S$) and $\Gamma (u)= (1,2)$.
\item $\mathcal{M}^0(G,3,3;I_{3}) \cup_{\Psi}^{\Gamma}
\langle w_1, w_2\rangle$, with $\Gamma(w_1)=(2,1,3,\theta)$ and
$\Gamma(w_2)=(2,3,\theta)(1)$,
$w_{2}w_{1}^{2}=w_{1}^{2}w_{2}=w_{1}^{3}=w_{2}w_{1}w_{2}= \theta$.
\item
 $\mathcal{M}^0(G,n,n;I_{n}) \cup_{\Psi}^{\Gamma} \langle v_1, v_2\rangle$, with
$$[\ldots, k, m, \ldots, k',m', \ldots]\sqsubseteq \Gamma(v_1), [\ldots, k, m', \ldots, k', m, \ldots] \sqsubseteq \Gamma(v_2)$$
for pairwise  distinct numbers $k, k', m$ and $m'$ between $1$ and
$n$, there do not exist distinct numbers $l_1$ and $l_2$ between $1$ and $n$ such that $(l_1,l_2)\subseteq \Gamma(x)$
for some  $x\in \langle v_1, v_2 \rangle$ 
and there do not exist pairwise distinct numbers $o_1, o_2$ and $o_3$ between $1$ and $n$ such that $(o_2,o_1,o_3,\theta)$ $\subseteq \Gamma(y_1)$, $(o_2,o_3,\theta)(o_1)\subseteq \Gamma(y_2)$  for some  $y_1,y_2\in \langle v_1, v_2 \rangle$.
\end{enumerate}

\end{cor}

{\it An example of a semigroup of type $U_4$ that is not \C.}\\
Consider the following semigroup
  $$S= \mathcal{M}^0(\{1_G,g \},3,3;I_{3})\cup^{\Gamma}_{id} \langle w,v
\rangle$$ with  $v^2=v^3$, $wv^2=wv$, $vw=v^2w$, $w^2=wvw=wv^2w$,
$vw^{2}=w^{2}v=w^{3}=vwv= \theta$,
$\Gamma(w) = (2,1,3,\theta)$ and $\Gamma(v) = (2,3,\theta)(1)$. 
Clearly,
 $\mathcal{M}^0(\{1_G\},3,3;I_{3}$ $)\cup^{\Gamma}_{id}
\langle w,v \rangle$ is a proper semigroup. The latter is \C\ and thus  $S$ is not \C.

{\it An example of a finite \C\ semigroup of type $U_5$  with trivial maximal subgroups in $M$.}\\
Let
$$S=\mathcal{M}^0(G,4,4;I_{4})\cup^{\Gamma}_{\Psi} \langle w,v
\rangle$$ with $G=\{1_G,g\}$ a cyclic group of order 2, $w^{2}=v^{2}=wv=vw=\theta,$
$$\Gamma(w) =
(4,1,\theta)(3,2,\theta) \mbox{ and }  \Gamma(v) = (4,2,\theta)(3,1,\theta),$$
$$\Psi(w)(4)=\Psi(w)(3)=\Psi(v)(4)=1,\;\;  \Psi(v)(3)=g \mbox{  and } \langle
w,v\rangle=\{w,v,\theta\}.$$
Since
$$\Gamma((1_{G};3,1))=\lambda_2(\Gamma((1_{G};3,1)),\Gamma((1_{G};4,2)),\Gamma(w),\Gamma(v))$$
and
$$\Gamma((1_{G};4,2))=\rho_2(\Gamma((1_{G};3,1)),\Gamma((1_{G};4,2)),\Gamma(w),\Gamma(v)),$$
we obtain that $S$ is not nilpotent.
Suppose that a subsemigroup $S'$
of $S$ is not
nilpotent. By Lemma~\ref{finite-nilpotent}, there exists a positive
integer $m$, distinct  elements $x, y\in S'$ and elements $w_{1}, w_{2}, \ldots,
w_{m}\in S'^1$ such that $x = \lambda_{m}(x, y, w_{1}, w_{2},
\ldots, w_{m})$ and  $y = \rho_{m}(x, y, w_{1}, w_{2}$ $, \ldots,
w_{m})$. As  $\langle w, v\rangle$ is nilpotent,
$$\{x, y, w_{1}, w_{2}, \ldots, w_{m}\} \cap \mathcal{M}^0(G,4,4;I_{4})
\neq \emptyset $$ and because $\mathcal{M}^0(G,4,4;I_{4})$ is an
ideal of $S$, $x$ and $y$ are nonzero elements in $\mathcal{M}^0(G,4,4;I_{4})$.
Since $S'$ is not nilpotent and $\mathcal{M}^0(G,4,4;I_{4})$ is
nilpotent, at least one element of the set $\{w_1, \ldots, w_m\}$
is not in $\mathcal{M}^0(G,4,4;I_{4})$. As before, without loss of generality, we may suppose that $w_1 \not
\in \mathcal{M}^0(G,4,4;I_{4})$. Write $x=(g_1 ;n_1,n_2)$ and
$y=(g_2;n_3,n_4)$, for some $1\leq n_1,n_2,n_3,n_4 \leq 4$ and $g_1,g_2
\in G$.

If $w_1 =w$ then $\{(n_1,n_2),(n_3,n_4)\}=\{(3,1),(4,2)\}$ and it
can be easily verified that $w_2=v$.

It also is easily verified that  $\mathcal{M}^0(\{1_G\},4,4;I_{4})$ is a subsemigroup of 
  $$\langle \Gamma((g_1 ;3,1)),\, \Gamma((g_2 ;4,2)),\,  \Gamma(w),\,  \Gamma(v)\rangle .$$
So,
$\mathcal{M}^0(\{1_G\},4,4;I_{4}) \subseteq \Gamma(S')$. Therefore, for every pair $1 \leq \alpha, \beta \leq 4$, there  exists an element $p \in S'$ such that $\Gamma(p)=(\beta,\alpha,\theta)$. As $\Gamma(w) =
(4,1,\theta)(3,2,\theta)$ and $\Gamma(v) = (4,2,\theta)(3,1,\theta)$ and $\langle w, v\rangle= \{w,v,\theta\}$, we obtain that there exists an element $h\in\{1_G,g\}$ such that $p=(h;\alpha,\beta)$.

If $S \neq S'$ then there exists an element $(k;i,j) \in \mathcal{M}^0(G,4,4;I_{4})$ such that $(k;i,j)\not\in S'$.
Now suppose that $(n_1,n_2)=(3,1)$ and $(n_3,n_4)=(4,2)$. Then $$v(g_1 ;3,1)v=(gg_1g ;1,3), \; \; w(g_2 ;4,2)w=(g_2 ;1,3), $$
$$v(g_1 ;3,1)w=(gg_1 ;1,4), \; \; w(g_2 ;4,2)v=(g_2 ;1,4). $$
Since $g_1,g_2 \in\{1_G,g\}$ we get that
both $(1_G;1,3)$ and $(g;1,3)$,  or both $(1_G;1,4)$ and $(g;1,4)$ are  in $S'$.
Suppose that both $(1_G;1,3)$ and $(g;1,3)$ are in $S'$. As proved above, there exist elements $k_{1},k_{2}\in G$ such that  $(k_1;i,1),\,  (k_2;3,j)\in S'$. Then $(k_1;i,1)(1_G;1,3)(k_2;3,j),\; (k_1;i,1)(g;1,3)(k_2;$ $3,j)\in S'$ and thus $(k_1k_2;i,j),\; (k_1gk_2;i,j)\in S'$. Since $k,k_1, k_2 \in \{1_G,g\}$ we get that  $(k;i,j)$ is in $S'$, a contradiction.
So, $S=S'$ in this case.


Similarly,  $(1_G;1,4),\; (g;1,4)\in S'$ leads to  $S=S'$. Hence, we have proved that $S=S'$ if $w_{1}=w$. If $w_1=v$, then one proves in an analogous manner that $S=S'$.  So, it follows that  $S$ is a \C\
semigroup of type $U_{5}$.


{\it An infinite class of finite \C\ semigroups of type $U_5$.}\\
Let $n\geq 5$ and consider
$\mathcal{M}^{0}(\{ e \},n,n;I_{n})$ as a subsemigroup of the full
transformation semigroup (see (\ref{elements-of-ideal})) on $\{
1,\ldots, n\} \cup \{ \theta \}$, i.e. we identify $(e;i,j)$ with
the cycle $(j,i,\theta )$ if $i\neq j$ and $(e;i,i)$ with the
permutation $(i)$. Let
$$Y_{n} = \mathcal{M}^0(\{e\},n,n;I_{n})\cup^{\Gamma}_{id} \langle w,v
\rangle$$
 and 
$$\Gamma(w)=(2,3,\theta)(4,1,\theta), \; \Gamma(v)=(2,1,\theta)(n,n-1,
\ldots,5,4,3,\theta).$$
It can be easily verified that
$$\Gamma(v^pw^q) = \Gamma(w^k) =\Gamma(v^l) = \theta \mbox{ for } p,q \geq 1,\;  k
\geq 2, \; l\geq n-2,$$
$$\Gamma(w^qv^p) = \theta \mbox{ for } q \geq 2, p \geq 1,$$
$$\Gamma(wv^p) = (p+4,1,\theta) \mbox{ for } n-4 \geq p \geq 1,
\Gamma(wv^p) = \theta \mbox{ for } p > n-4 \mbox{ and}$$
$$\Gamma(awv^p) = \theta \mbox{ for } p \geq 1, a \in \langle w,
v\rangle.$$ Because $\abs{\Gamma^{-1}(\theta )}=1$ we thus obtain that  $v^pw^q = w^k =v^l = \theta$ for $p,q \geq 1,
k \geq 2,\; l\geq n-2,\;  w^qv^p = \theta$ for $q \geq 2, p \geq 1$ and $awv^p =
\theta $ for $p \geq 1, a \in \langle w, v\rangle$. So, $$\langle w,v \rangle =\{ w,v,\ldots,v^{n-3},wv, \ldots, wv^{n-4}, \theta\}$$ and
clearly $\langle w, v\rangle^n=\{\theta\}$. Therefore $\langle w, v\rangle$ is nilpotent.

We claim that $Y_n$ is \C.
To prove this, suppose that  $Y$ is a subsemigroup of $Y_{n}$ that is  not nilpotent. We need to prove that $Y=Y_{n}$.
As before, there exists a positive
integer $m$, distinct elements $x, y\in \mathcal{M}^0(\{e\},n,n;I_{n})$
and elements $w_{1},
w_{2}, \ldots, w_{m}\in Y^1$ with $w_{1}\not\in \mathcal{M}^0(\{e\},n,n;I_{n})$ such that $x = \lambda_{m}(x, y,
w_{1}, w_{2}, \ldots, w_{m})$, $y = \rho_{m}(x, y, w_{1}, w_{2}$
$, \ldots, w_{m})$.
Write $x=(e;n_1,n_2)$,
$y=(e;n_3,n_4)$ for some $1\leq n_1,n_2,n_3,n_4 \leq n$. Since $x
\neq y$, $(n_1,n_2) \neq (n_3,n_4)$.

Since $\Gamma(xw_1y)$ and $\Gamma(yw_1x)$ are nonzero, $\Gamma(w_1)(n_3)=n_2$ and
$\Gamma(w_1)(n_1)=n_4$. Now as $\Gamma(wv^p) = (p+4,1,\theta) \mbox{ for } n-4 \geq p \geq 1$ and $(n_1,n_2)\neq (n_3,n_4)$, $w_1 \notin \{wv, \ldots, wv^{n-4}\}$. Similarly $w_2 \notin \{wv, \ldots, wv^{n-4}\}$.
So, $w_{1}=w$ or  $w_{1}=v^{k}$ for some $n-3\geq k\geq 1$.

Suppose that $w_1 = v^k$ for some $1\leq k \leq n-3$.  Since $\Gamma(xw_1y)$ and
$\Gamma(yw_1x)$ are nonzero, $\Gamma(v^k)(n_3)$ $=n_2$ and
$\Gamma(v^k)(n_1)=n_4$. Hence $\lambda_1=(e;n_1,n_4),
\rho_1=(e;n_3,n_2)$. If $w_2=w$ then $\Gamma(w)(n_1)=n_2$ and
$\Gamma(w)(n_3)=n_4$. Since $(n_1,n_2) \neq (n_3,n_4)$ and $\Gamma(w)=(2,3,\theta)(4,1,\theta)$, $n_1=4,
n_2=1, n_3=2, n_4=3$ or $n_1=2, n_2=3, n_3=4, n_4=1$. Now as $\Gamma(v^k)(n_3)=n_2$,
$\Gamma(v^k)(n_1)=n_4$ and $\Gamma(v^k)(2)=\theta$ for $k > 1$,
$v^k=v$ and thus $Y=Y_n$. So we may assume that there exists $1\leq l\leq
n$ such that $w_2=v^l$. Similarly we have $\Gamma(v^l)(n_1)=n_2$ and
$\Gamma(v^l)(n_3)=n_4$. It can be easily verified that $n_4 = n_1 -
k= n_3 -l , n_2= n_1 - l =n_3 - k$. Then $k-l=l-k$ and thus $k=l$.
Hence $n_4=n_2, n_1=n_3$, a contradiction.

Finally suppose that $w_1=w$. As $\Gamma(w_1)(n_3)=n_2$,
$\Gamma(w_1)(n_1)=n_4$, $\Gamma(w)=(2,3,\theta)\, (4,1,\theta)$ and $(n_1,n_2) \neq (n_3,n_4)$, $\{x, y\} = \{(e;4,3),(e;2,1)\}$
and thus
$\rho_1, \lambda_1 \in \{(e;2,3),(e;4,1)\}$.
Now as $\Gamma(\rho_1w_2\lambda_1)$ and $\Gamma(\lambda_1w_2\rho_1)$ are nonzero, it follows that $(2,1,\theta)\subseteq
\Gamma(w_2),(4,3,\theta)\subseteq \Gamma(w_2)$. Since $\Gamma(v^k)(2)=\theta$ for $k > 1$ and $\Gamma(w)(2)=3$, one then obtains that $w_2=v$.
Therefore $Y=Y_{n}$.
It follows that indeed $Y_n$ is \C.



\end{document}